\documentclass[small]{article}
\usepackage{amsmath,amstext,amsthm,amsfonts}
\usepackage{amssymb,latexsym}
\usepackage{txfonts, pxfonts}
\usepackage[english]{babel}
\usepackage[pdftex]{graphicx}
\numberwithin{equation}{section}
\newtheorem{theorem}{\textbf{Theorem}}[section]
\newtheorem{proposition}[theorem]{\textbf{Proposition}}
\newtheorem{lemma}[theorem]{\textbf{Lemma}}

\theoremstyle{definition}
\newtheorem{definition}[theorem]{\textbf{Definition}}
\theoremstyle{remark}

\def\e{\epsilon}

\def\R{\mathbb{R}}
\def\Rn{{\mathbb{R}}^n_+}
\def\Rnn{{\mathbb{R}}^{n-1}}
\def\d{\partial}

\def\B{B^+_{\delta}}

\def\U{U_{\epsilon}}

\def\Be{B^+_{\delta\e^{-1}}}

\def\ba{\begin{align}}
\def\ea{\end{align}}
\def\bp{\begin{proof}}
\def\ep{\end{proof}}
\def\Z{((1+y_n)^2+|\bar{y}|^2)}
\def\Zt{((1+y_n)^2+|\bar{y}|^2)}

\title{An existence theorem of conformal scalar-flat metrics
on manifolds with boundary}

\author{S\'ergio Almaraz}

\begin{document}

\maketitle
\begin{abstract}
Let $(M,g)$ be a compact Riemannian manifold with boundary. This paper addresses the Yamabe-type problem of finding a conformal scalar-flat metric on $M$, which has the boundary as a constant mean curvature hypersurface. When the boundary is umbilic, we prove an existence theorem that finishes some remaining cases of this problem.
\end{abstract}

\section{Introduction}

In 1992, J. Escobar (\cite{escobar3}) studied the following Yamabe-type problem, for manifolds with boundary:

\vspace{0.2cm}
\noindent
YAMABE PROBLEM: {\it{Let $(M^n,g)$ be a compact Riemannian manifold of dimension $n\geq 3$ with boundary $\d M$. Is there a scalar-flat metric on $M$, which is conformal to $g$ and has $\d M$ as a constant mean curvature hypersurface?}}
\vspace{0.1cm}

In dimension two, the classical Riemann mapping theorem says that any simply connected, proper domain of the plane is conformally diffeomorphic to a disk. This theorem is false in higher dimensions since the only bounded open subsets of $\R^n$, for $n\geq 3$, that are conformally diffeomorphic to Euclidean balls are the Euclidean balls themselves. The Yamabe-type problem proposed by Escobar can be viewed as an extension of the Riemann mapping theorem for higher dimensions.

%
In analytical terms, this problem corresponds to finding a positive solution to 
\begin{align}\label{eq:u'}
\begin{cases}
L_{g}u=0,&\text{in}\:M,
\\
B_{g}u+Ku^{\frac{n}{n-2}}=0,&\text{on}\:\partial M,
\end{cases}
\end{align}
for some constant $K$, where $L_g=\Delta_g-\frac{n-2}{4(n-1)}R_g$ is the conformal Laplacian and $B_g=\frac{\d}{\d\eta}-\frac{n-2}{2}h_g$. Here, $\Delta_g$ is the Laplace-Beltrami operator, $R_g$ is the scalar curvature, $h_g$ is the mean curvature of $\d M$ and  $\eta$ is the inward unit normal vector to $\d M$. 

The solutions of the equations (\ref{eq:u'}) are the critical points of the functional 
$$
Q(u)=\frac{\int_M|\nabla_gu|^2+\frac{n-2}{4(n-1)}R_gu^2dv_g+\frac{n-2}{2}\int_{\d M}h_gu^2d\sigma_g}
{\left(\int_{\d M}u^\frac{2(n-1)}{n-2}d\sigma_g\right)^{\frac{n-2}{n-1}}}\,,
$$
where $dv_g$ and $d\sigma_g$ denote the volume forms of $M$ and $\d M$, respectively. 
Escobar  introduced the conformally invariant Sobolev quotient
$$
Q(M,\d M)=\inf\{Q(u);\:u\in C^1(M), u\nequiv 0 \:\text{on}\: \d M\}
$$
and proved that it satisfies $Q(M,\d M)\leq Q(B^n,\d B)$. Here, $B^n$ denotes the unit ball in $\R^n$ endowed with the Euclidean metric. 

Under the hypothesis that $Q(M, \d M)$ is finite (which is the case when $R_g\geq 0$), he also showed that the strict inequality 
\begin{equation}\label{ineq:Q}
Q(M,\d M)< Q(B^n,\d B)
\end{equation}
implies the existence of a minimizing solution of the equations (\ref{eq:u'}). 

\vspace{0.2cm}
\noindent{\bf{Notation.}}
In the rest of this work, $(M^n,g)$ will denote a compact Riemannian manifold of dimension $n\geq 3$ with boundary $\d M$ and finite Sobolev quotient $Q(M,\d M)$. 
\vspace{0.1cm}

In \cite{escobar3}, Escobar proved the following existence result:
\begin{theorem}\label{thm:escobar}
(J. Escobar) Assume that one of the following conditions holds:

\vspace{0.2cm}
\noindent
(1) $n\geq 6$ and $M$ has a nonumbilic point on $\d M$;
\\
(2) $n\geq 6$, $M$ is locally conformally flat and $\d M$ is umbilic;
\\
(3) $n=4$ or $5$ and $\d M$ is umbilic;
\\
(4) $n=3$.
\vspace{0.1cm}

\noindent
Then  $Q(M,\d M)< Q(B^n,\d B)$ and there is a minimizing solution to the equations (\ref{eq:u'}).
\end{theorem} 
\noindent 
The proof for $n=6$ under the condition (1) appeared later, in \cite{escobar4}. 

Further existence results were obtained by F. Marques in  \cite{coda1} and \cite{coda2}. Together, these results can be stated as follows:
\begin{theorem}\label{thm:coda}
(F. Marques)  Assume that one of the following conditions holds:

\vspace{0.2cm}
\noindent
(1) $n\geq 8$, $\bar{W}(x)\neq 0$ for some $x\in\d M$ and $\d M$ is umbilic;
\\
(2) $n\geq 9$, $W(x)\neq 0$ for some $x\in\d M$ and $\d M$ is umbilic;
\\
(3) $n=4$ or $5$ and $\d M$ is not umbilic.
\vspace{0.1cm}

\noindent
Then  $Q(M,\d M)< Q(B^n,\d B)$ and there is a minimizing solution to the equations (\ref{eq:u'}).
\end{theorem} 
\noindent
Here, $W$ denotes the Weyl tensor of $M$ and $\bar{W}$ the Weyl tensor of $\d M$.

Our main result deals with the remaining dimensions $n=6,7$ and  $8$ when the boundary is umbilic and $W\neq 0$ at some boundary point: 
\begin{theorem}\label{exist:thm}
Suppose that $n=6,7$ or $8$, $\d M$ is umbilic and $W(x)\neq 0$ for some $x\in\d M$. Then  $Q(M,\d M)< Q(B^n,\d B)$ and there is a minimizing solution to the equations (\ref{eq:u'}).
\end{theorem}
These cases are similar to the case of dimensions $4$ and $5$ when the boundary is not umbilic, studied in \cite{coda2}. 

Other works concerning conformal deformation on manifolds with boundary include \cite{ahmedou}, \cite{ambrosetti-li-malchiodi}, \cite{ayed-mehdi-ahmedou}, \cite{brendle0}, \cite{djadli-malchiodi-ahmedou1}, \cite{djadli-malchiodi-ahmedou2}, \cite{escobar2}, \cite{escobar5}, \cite{escobar-garcia}, \cite{ahmedou-felli1}, \cite{ahmedou-felli2}, \cite{han-li} and \cite{han-li2}.

We will now discuss the strategy in the proof of Theorem \ref{exist:thm}. We assume that $\d M$ is umbilic and choose $x_0\in \d M$ such that $W(x_0)\neq 0$. Our proof is explicitly based on constructing a test function $\psi$ such that
\begin{equation}\label{des:psi:Q}
Q(\psi)<Q(B^n,\d B)\,.
\end{equation}
The function $\psi$ has support in a small half-ball around the point $x_0$.
The usual strategy in this kind of problem (which goes back to Aubin in \cite{aubin1}) consists in defining the function $\psi$, in the small half-ball, as one of the standard entire solutions to the corresponding Euclidean equations. In our context those are 
\begin{equation}\label{Ue}
U_{\e}(x)=\left(\frac{\e}{x_1^2+...+x_{n-1}^2+(\e+x_n)^2}\right)^{\frac{n-2}{2}}\,.
\end{equation}
where $x=(x_1,...,x_n)$, $x_n\geq 0$.

The next step would be to expand the quotient of $\psi$ in powers of $\e$ and, by exploiting the local geometry around $x_0$, show that the inequality (\ref{des:psi:Q}) holds if $\e$ is small. In order to simplify the asymptotic analysis, we use conformal Fermi coordinates centered at $x_0$. This concept, introduced in \cite{coda1}, plays the same role the conformal normal coordinates (see \cite{lee-parker}) did in the case of manifolds without boundary. 

When $n\geq 9$, the strict inequality (\ref{des:psi:Q}) was proved in \cite{coda1}. 
The difficulty arises because, when $3\leq n\leq 8$, the first correction term in the expansion does not have the right sign. When $3\leq n\leq 5$, Escobar proved the strict inequality by applying the Positive Mass Theorem, a global construction originally due to Schoen (\cite{schoen1}). 
This argument does not work when $6\leq n\leq 8$ because the metric is not sufficiently flat around the point $x_0$.

As we have mentioned before, the situation under the hypothesis of Theorem \ref{exist:thm} is much similar to the cases of dimensions $4$ and $5$ when the boundary is not umbilic, solved by Marques in \cite{coda2}. 
As he pointed out, the test functions $U_{\e}$ are not optimal in these cases but the problem is still local. This kind of phenomenon does not appear in the classical solution of the Yamabe problem for manifolds without boundary. However, perturbed test functions have already been used in the works of Hebey and Vaugon (\cite{hebey-vaugon}), Brendle (\cite{brendle3}) and Khuri, Marques and Schoen (\cite{khuri-marques-schoen}).

In order to prove the inequality (\ref{des:psi:Q}), inspired by the ideas of Marques, we introduce
\begin{equation}\notag
\phi_{\e}(x)=\e^{\frac{n-2}{2}}R_{ninj}(x_0)x_ix_jx_n^2\left(x_1^2+...+x_{n-1}^2+(\e+x_n)^2\right)^{-\frac{n}{2}}\,.
\end{equation}
Our test function $\psi$ is defined as 
$
\psi=U_{\e}+\phi_{\e}
$
around $x_0\in\d M$.
%

In section 2 we write expansions for the metric $g$ in Fermi coordinates and discuss the concept of conformal Fermi coordinates. In section 3 we prove Theorem \ref{exist:thm} by estimating $Q(\psi)$.
\\\\
{\bf{Notations.}}

Throughout this work we will make use of the index notation for tensors, commas denoting covariant differentiation. We will adopt the summation convention whenever confusion is not possible. When dealing with Fermi coordinates, we will use indices $1\leq i,j,k,l,m,p,r,s\leq n-1$ and $1\leq a,b,c,d\leq n$.  Lines under or over an object mean the restriction of the metric to the boundary is involved.

We 
set $\det g=\det g_{ab}$. We will denote by $\nabla_g$ or $\nabla$ the covariant derivative and by $\Delta_g$ or $\Delta$ the Laplacian-Beltrami operator. The full curvature tensor will be denoted by $R_{abcd}$,  the Ricci tensor by $R_{ab}$ and the scalar curvature by $R_g$ or $R$. The second fundamental form of the boundary will be denoted by $h_{ij}$ and the mean curvature,  $\frac{1}{n-1}tr (h_{ij})$, by $h_g$ or $h$. By $W_{abcd}$ we will denote the Weyl tensor.

By $\Rn$ we will denote the half-space $\{x=(x_1,...,x_n)\in \R^n;\:x_n\geq 0\}$. If $x\in\Rn$ we set $\bar{x}=(x_1,...,x_{n-1})\in\R^{n-1}\cong \d\R^{n}$. We will denote by $B_{\delta}^+(0)$ (or $B^+_{\delta}$ for short) the half-ball $B_{\delta}(0)\cap\Rn$, where $B_{\delta}(0)$ is the Euclidean open  ball of radius $\delta>0$ centered at the origin of $\R^n$.  
Given a subset $\mathcal{C}\subset \Rn$, we set $\d^+\mathcal{C}=\d \mathcal{C}\cap \Rn$ and $\d '\mathcal{C}= \mathcal{C}\cap\d \Rn$.

The volume forms of $M$ and $\d M$ will be denoted by $dv_g$ and $d\sigma_g$, respectively.
The n-dimensional sphere of radius $r$ centered at the origin of $\R^{n+1}$ will be denoted by $S^{n}_r$. By $\sigma_{n}$ we will denote the volume of the n-dimensional unit sphere $S^n_1$. 

For $\mathcal{C}\subset M$, we define the energy of a function $u$ in $\mathcal{C}$ by 
$$
E_{\mathcal{C}}(u)
=\int_{\mathcal{C}}\left(|\nabla_g u|^2+\frac{n-2}{4(n-1)}R_gu^2\right)dv_g
+\frac{n-2}{2}\int_{\d '\mathcal{C}}h_gu^2d\sigma_g\,.
$$

\bigskip\noindent
{\bf{Acknowledgements.}}
The content of this paper is a part of the author's doctoral thesis 
(\cite{almaraz1}) at IMPA. 
The author would like to express his gratitude to his advisor Prof. Fernando C. Marques for numerous mathematical conversations and constant 
encouragement. While the author was at IMPA, he was fully supported by CNPq-Brazil.


\section{Coordinate expansions for the metric}

In this section we will write expansions for the metric $g$ in Fermi coordinates. We will also discuss the concept of conformal Fermi coordinates, introduced by Marques in \cite{coda1}, that will simplify the computations of the next section. The conformal Fermi coordinates play the same role that the conformal normal coordinates (see \cite{lee-parker}) did in the case of manifolds without boundary. The results of this section are basically proved on pages 1602-1609 and 1618 of \cite{coda1}. 
\begin{definition}
Let $x_0\in\d M$. We choose geodesic normal coordinates $(x_1,...,x_{n-1})$ on the boundary, centered at $x_0$.
We say that $(x_1,...,x_n)$, for small $x_n\geq 0$, are the Fermi coordinates (centered at $x_0$) of the point $\exp_{x}(x_n\eta(x))\in M$. Here, we denote by $\eta(x)$ the inward unit vector normal to $\d M$ at $x\in\d M$.
\end{definition}

It is easy to see that in these coordinates $g_{nn}\equiv 1$ and $g_{jn}\equiv 0$, for $j=1,...,n-1$.

We fix $x_0\in\d M$.
The existence of conformal Fermi coordinates is stated as follows:
\begin{proposition}\label{conf:fermi:thm}
For any given integer $N\geq 1$ there is a metric $\tilde{g}$, conformal to $g$, such that in $\tilde{g}$-Fermi coordinates centered at $x_0$
$$
\det \tilde{g}(x)=1+O(|x|^N)\,.
$$
Moreover, $h_{\tilde{g}}(x)=O(|x|^{N-1})$.
\end{proposition}
The first statement of Proposition \ref{conf:fermi:thm} is Proposition 3.1 of \cite{coda1}. The second one follows from the equation
\begin{equation}\notag
h_g=\frac{-1}{2(n-1)}g^{ij}g_{ij,\,n}=\frac{-1}{2(n-1)}(\log\det g)_{,\,n}\,.
\end{equation}
The next three lemmas will also be used in the computations of the next section.
\begin{lemma}\label{exp:g}
Suppose that $\d M$ is umbilic. Then, in conformal Fermi coordinates centered at $x_0$, $h_{ij}(x)=O(|x|^N)$, where $N$ can be taken arbirarily large, and
\begin{align}
g^{ij}(x)=\;&\delta_{ij}
+\frac{1}{3}\bar{R}_{ikjl}x_kx_l+R_{ninj}x_n^2
+\frac{1}{6}\bar{R}_{ikjl;\,m}x_kx_lx_m
+R_{ninj;\,k}x_n^2x_k+\frac{1}{3}R_{ninj;\,n}x_n^3\notag
\\
&\hspace{1.5cm}+\left(\frac{1}{20}\bar{R}_{ikjl;\,mp}+\frac{1}{15}\bar{R}_{iksl}\bar{R}_{jmsp}\right)x_kx_lx_mx_p\notag
\\
&\hspace{1.4cm}+\left(\frac{1}{2}R_{ninj;\,kl}+\frac{1}{3}Sym_{ij}(\bar{R}_{iksl}R_{nsnj})\right)x_n^2x_kx_l\notag
\\
&\hspace{0.5cm}+\frac{1}{3}R_{ninj;\,nk}x_n^3x_k
+\left(\frac{1}{12}R_{ninj;\,nn}+\frac{2}{3}R_{nins}R_{nsnj}\right)x_n^4+\;O(|x|^5)\,.\notag
\end{align}
\noindent
Here, every coefficient is computed at $x_0$.
\end{lemma}
\begin{lemma}\label{conseq:fermi:conf'}
Suppose that $\d M$ is umbilic. Then, in conformal Fermi coordinates centered at $x_0$,

\vspace{0.2cm}
\noindent
(i) $\bar{R}_{kl}=Sym_{klm}(\bar{R}_{kl;\,m})=0$;
\\
(ii) $R_{nn}=R_{nn;\,k}=Sym_{kl}(R_{nn;\,kl})=0$;
\\
(iii) $R_{nn;\,n}=0$;
\\
(iv) $Sym_{klmp}(\frac{1}{2}\bar{R}_{kl;\,mp}+\frac{1}{9}\bar{R}_{ikjl}\bar{R}_{imjp})=0$;
\\
(v) $R_{nn;\,nk}=0$;
\\
(vi) $R_{nn;\,nn}+2(R_{ninj})^2=0$;
\\
(vii) $R_{ij}=R_{ninj}$;
\\
(viii) $R_{ijkn}=R_{ijkn;\,j}=0$;
\\
(ix) $R=R_{,\,j}=R_{,\,n}=0$;
\\
(x) $R_{,\,ii}=-\frac{1}{6}(\bar{W}_{ijkl})^2$;
\\
(xi) $R_{ninj;\,ij}=-\frac{1}{2}R_{;\,nn}-(R_{ninj})^2$;
\vspace{0.1cm}

\noindent
where all the quantities are computed at $x_0$.
\end{lemma}
The idea to prove the items (i),...,(vi) of Lemma \ref{conseq:fermi:conf'} is to express $g_{ij}$ as the exponencial of a matrix $A_{ij}$. Then we just observe that $\text{trace}(A_{ij})=O(|x|^N)$ for any integer $N$ arbitrarily large. The items (vii)...(xi) are applications of the Gauss and Codazzi equations and the Bianchi identity. We should mention that the item (x) uses the fact that Fermi coordinates are normal on the boundary.
\begin{lemma}\label{anul:W}
Suppose that $\d M$ is umbilic. Then, in conformal Fermi coordinates centered at $x_0\in \d M$, $W_{abcd}(x_0)=0$ if and only if $R_{ninj}(x_0)=\bar{W}_{ijkl}(x_0)=0$.
\end{lemma}

For the sake of the reader we include the proof of Lemma \ref{anul:W} here.
\begin{proof}[Proof of Lemma \ref{anul:W}] 
Recall that the  Weyl tensor is defined by
\ba\label{def:weyl}
W_{abcd}&=R_{abcd}-\frac{1}{n-2}\left(R_{ac}g_{bd}-R_{ad}g_{bc}+R_{bd}g_{ac}-R_{bc}g_{ad}\right)\notag
\\
&\hspace{1cm}+\frac{R}{(n-2)(n-1)}\left(g_{ac}g_{bd}-g_{ad}g_{bc}\right)\,.
\end{align}

By the symmetries of the Weyl tensor, $W_{nnnn}=W_{nnni}=W_{nnij}=0$. By the identity (\ref{def:weyl}) and Lemma \ref{conseq:fermi:conf'} (viii), $W_{nijk}(x_0)=0$. From the identity (\ref{def:weyl}) again and from Lemma \ref{conseq:fermi:conf'} (ii), (vii), (ix), 
$$
W_{ninj}=\frac{n-3}{n-2}R_{ninj}
$$
and
$$
W_{ijkl}=\bar{W}_{ijkl}-\frac{1}{n-2}
\left(R_{nink}g_{jl}-R_{ninl}g_{jk}+R_{njnl}g_{ik}-R_{njnk}g_{il}\right)
$$
at $x_0$. In the last equation we also used the Gauss equation and Lemma \ref{exp:g}. Now the result follows from the above equations.
\ep


\section{Estimating the Sobolev quotient}

In this section, we will prove Theorem \ref{exist:thm} by constructing a function $\psi$ such that
$$
Q(\psi)<Q(B^n,\d B)\,.
$$

We first recall that the positive number $Q(B^n, \d B)$ also appears as the best constant in the following Sobolev-trace inequality:
$$
\left(\int_{\d\Rn}|u|^{\frac{2(n-1)}{n-2}}d\bar{x}\right)^{\frac{n-2}{n-1}}
\leq \frac{1}{Q(B^n,\d B)}\int_{\Rn}|\nabla u|^2dx\,,
$$
for every $u\in H^1(\Rn)$.
It was proven by Escobar (\cite{escobar0}) and independently by Beckner (\cite{beckner}) that the equality is achieved by the functions $U_{\e}$, defined in (\ref{Ue}). They are solutions to the boundary-value problem
\begin{equation}\label{eq:Ue}
\begin{cases}
\Delta \U = 0\,,&\text{in}\:\mathbb{R}_+^n\,,\\
\frac{\partial \U}{\partial y_n}+(n-2)\U^{\frac{n}{n-2}}=0\,,&\text{on}\:\partial\mathbb{R}_+^n\,.
\end{cases}
\end{equation}
One can check, using integration by parts, that
$
\int_{\Rn}|\nabla U_{\e}|^2dx=(n-2)\int_{\d\Rn}U_{\e}^{\frac{2(n-1)}{n-2}}dx
$
and also that
\begin{equation}\label{eq:Q:Ue}
Q(B^n,\d B)=(n-2)\left(\int_{\d\Rn}U_{\e}^{\frac{2(n-1)}{n-2}}dx\right)^{\frac{1}{n-1}}\,.
\end{equation}

\noindent{\bf{Assumption.}}
In the rest of this work we will assume that $\d M$ is umbilic and there is a point $x_0\in \d M$ such that $W(x_0)\neq 0$.
\vspace{0.1cm}

Since the Sobolev quotient $Q(M,\d M)$ is a conformal invariant, we can use conformal Fermi coordinates centered at $x_0$.

\vspace{0.2cm}
\noindent{\bf{Convention.}}
In what follows, all the curvature terms are evaluated at $x_0$.  We fix conformal Fermi coordinates centered at $x_0$ and work in a half-ball $B_{2\delta}^+=B_{2\delta}^+(0)\subset\Rn$. 
\vspace{0.1cm}

In particular, for any $N$ arbitrarily large, we can write the volume element $dv_g$ as
\begin{equation}\label{eq:vol}
dv_g=(1+O(|x|^N))dx\,.
\end{equation}

In many parts of the text we will use the fact that, for any homogeneous polynomial $p_k$ of degree $k$, 
\ba\label{rel:int:pol'}
\int_{S_r^{n-2}}p_k=\frac{r^2}{k(k+n-3)}\int_{S_r^{n-2}}\Delta p_k\,.
\end{align}

We will now construct the test function $\psi$. Set
\begin{equation}\label{def:phi:e}
\phi_{\e}(x)=\e^{\frac{n-2}{2}}AR_{ninj}x_ix_jx_n^2\left((\e+x_n)^2+|\bar{x}|^2\right)^{-\frac{n}{2}}\,,
\end{equation}
for $A\in\R$ to be fixed later, and
\begin{equation}
\phi(y)=AR_{ninj}y_iy_jy_n^2\left((1+y_n)^2+|\bar{y}|^2\right)^{-\frac{n}{2}}\,.
\end{equation}
Thus, $\phi_{\e}(x)=\e^{2-\frac{n-2}{2}}\phi(\e^{-1}x)$. Set $U=U_1$. Thus, $\U(x)=\e^{-\frac{n-2}{2}}U(\e^{-1}x)$. Note that $\U(x)+\phi_{\e}(x)=(1+O(|x|^2))\U(x)$. Hence, if $\delta$ is sufficiently small,
$$
\frac{1}{2}\U\leq \U+\phi_{\e}\leq 2\U\,,\:\:\:\:\text{in}\:B_{2\delta}^+\,.
$$ 

Let $r\mapsto \chi(r)$ be a smooth cut-off function satisfying $\chi(r)=1$ for $0\leq r\leq \delta$, $\chi(r)=0$ for $r\geq 2\delta$, $0\leq \chi \leq 1$ and $|\chi '(r)|\leq C\delta^{-1}$ if $\delta\leq r\leq 2\delta$.
Our test function is defined by
$$
\psi(x)=\chi(|x|)(U_{\e}(x)+\phi_{\e}(x))\,.
$$

\subsection{Estimating the energy of $\psi$}

The energy of $\psi$ is given by 
\begin{align}\label{exist:1}
E_M(\psi)
&=\int_{M}\left(|\nabla_g\psi|^2+\frac{n-2}{4(n-1)}R_g\psi^2\right)dv_g
+\frac{n-2}{2}\int_{\d M}h_g\psi^2d\sigma_g\notag
\\
&=E_{\B}(\psi)+E_{B^+_{2\delta}\backslash\B}(\psi)\,.
\end{align}

Observe that 
$$
|\nabla_g \psi|^2\leq C|\nabla\psi|^2\leq
C|\nabla\chi|^2(\U+\phi_{\e})^2+C\chi^2|\nabla(\U+\phi_{\e})|^2\,.
$$
Hence,
\begin{align}
E_{B^+_{2\delta}\backslash B_{\delta}^+}(\psi)&\leq 
C\int_{B^+_{2\delta}\backslash B_{\delta}^+}|\nabla\chi|^2\U^2dx
+C\int_{B^+_{2\delta}\backslash B_{\delta}^+}\chi^2|\nabla\U|^2dx\notag
\\
&\hspace{0.5cm}+C\int_{B^+_{2\delta}\backslash B_{\delta}^+}R_g\U^2dx
+C\int_{\d 'B^+_{2\delta}\backslash \d 'B_{\delta}^+}h_g\U^2d\bar{x}\,,\notag
\end{align}
Thus,
\begin{equation}\label{exist:2}
E_{B^+_{2\delta}\backslash B_{\delta}^+}(\psi)\leq C\e^{n-2}\delta^{2-n}\,.
\end{equation}

The first term in the right hand side of  (\ref{exist:1}) is
\ba\label{exist:3}
E_{\B}(\psi)
&=E_{\B}(\U+\phi_{\e})\notag
\\
&=\int_{\B}\left\{|\nabla_g(\U+\phi_{\e})|^2+\frac{n-2}{4(n-1)}R_g(\U+\phi_{\e})^2\right\}dv_g\notag
\\
&\hspace{0.5cm}+\frac{n-2}{2}\int_{\d 'B^+_{\delta}}h_g(\U+\phi_{\e})^2d\sigma_g\notag
\\
&=\int_{\B}|\nabla(\U+\phi_{\e})|^2dx+C\e^{n-2}\delta
\\
&\hspace{0.5cm}+\int_{\B}(g^{ij}-\delta^{ij})\d_i(\U+\phi_{\e})\d_j(\U+\phi_{\e})dx
+\frac{n-2}{4(n-1)}\int_{\B}R_g(\U+\phi_{\e})^2dx\,.\notag
\end{align}
Here, we used the identity (\ref{eq:vol}) for the volume term and Proposition \ref{conf:fermi:thm} for the integral envolving $h_g$.

Now, we will handle each of the three integral terms in the right hand side of (\ref{exist:3}) in the next three lemmas. 
\begin{lemma}\label{lema1}
We have,
\ba
\int_{\B}|\nabla(\U+\phi_{\e})|^2dx
&\leq Q(B^n,\d B^n)\left(\int_{\d M}\psi^{\frac{2(n-1)}{n-2}}dx\right)^{\frac{n-2}{n-1}}
+C\e^{n-2}\delta^{2-n}\notag
\\
&-\frac{4}{(n+1)(n-1)}\e^4A^2(R_{ninj})^2
\int_{B_{\delta\e^{-1}}}\frac{y_n^2|\bar{y}|^{4}}{\Z^n}dy\notag
\\
&+\frac{8n}{(n+1)(n-1)}\e^4A^2(R_{ninj})^2
\int_{B_{\delta\e^{-1}}^+}\frac{y_n^3|\bar{y}|^{4}}{\Z^{n+1}}dy\notag
\\
&+\frac{12n}{(n+1)(n-1)}\e^4A^2(R_{ninj})^2
\int_{B_{\delta\e^{-1}}^+}\frac{y_n^4|\bar{y}|^{4}}{\Z^{n+1}}dy\notag
\end{align}
\end{lemma}
\bp    
Since $R_{nn}=0$ (see Lemma \ref{conseq:fermi:conf'}(ii)),
$
\int_{S_r^{n-2}}R_{ninj}y_iy_j d\sigma_r(y)=0\,.
$
Thus, we see that
\begin{equation}\label{lema1:1}
\int_{\B}|\nabla(\U+\phi_{\e})|^2dx
=\int_{\B}|\nabla\U|^2dx+\int_{\B}|\nabla\phi_{\e}|^2dx\,.
\end{equation}

Integrating by parts equations (\ref{eq:Ue}) and using the identity (\ref{eq:Q:Ue}) we obtain
$$\int_{\B}|\nabla\U|^2dx\leq 
Q(B^n,\d B^n)\left(\int_{\d '\B}\U^{\frac{2(n-1)}{n-2}}dx\right)^{\frac{n-2}{n-1}}
\leq Q(B^n,\d B^n)\left(\int_{\d M}\psi^{\frac{2(n-1)}{n-2}}dx\right)^{\frac{n-2}{n-1}}\,.$$
In the first inequality above we used the fact that $\frac{\d \U}{\d\eta}>0$ on $\d^+B^+_{\delta}$, where $\eta$ denotes the inward unit normal vector . In the second one we used the fact that $\phi_{\e}=0$ on $\d M$.

For the second term in the right hand side of (\ref{lema1:1}), 
an integration by parts plus a change of variables gives
$$\int_{\B}|\nabla\phi_{\e}|^2dx\leq -\e^4\int_{B^+_{\e^{-1}\delta}}(\Delta\phi)\phi dy+C\e^{n-2}\delta^{2-n}\,,$$
since 
$\int_{\d '\B}\frac{\d\phi_{\e}}{\d x_n}\phi_{\e} d\bar{x}=0$
and the term $\e^{n-2}\delta^{2-n}$ comes from the integral over $\d^+\B$.
\\\\
{\it{Claim.}}
The function  $\phi$ satisfies
\ba
\Delta \phi(y)&=2AR_{ninj}y_iy_j\Z^{-\frac{n}{2}}-4nAR_{ninj}y_iy_jy_n\Z^{-\frac{n+2}{2}}\notag
\\
&\hspace{1cm}-6nAR_{ninj}y_iy_jy_n^2\Z^{-\frac{n+2}{2}}\,.\notag
\end{align}

In order to prove the Claim we set $Z(y)=\Z$. Since $R_{nn}=0$, 
\ba
\Delta(R_{ninj}y_iy_jy_n^2Z^{-\frac{n}{2}})
&=\Delta(R_{ninj}y_iy_jy_n^2)Z^{-\frac{n}{2}}+R_{ninj}y_iy_jy_n^2\Delta(Z^{-\frac{n}{2}})\notag
\\
&\hspace{0.5cm}+2\d_k(R_{ninj}y_iy_jy_n^2)\d_k(Z^{-\frac{n}{2}})
+2\d_n(R_{ninj}y_iy_jy_n^2)\d_n(Z^{-\frac{n}{2}})\notag
\\
&=2R_{ninj}y_iy_jZ^{-\frac{n}{2}}+2nR_{ninj}y_iy_jy_n^2Z^{-\frac{n+2}{2}}\notag
\\
&\hspace{0.5cm}-4nR_{ninj}y_iy_jy_n^2Z^{-\frac{n+2}{2}}
-4nR_{ninj}y_iy_jy_n(y_n+1)Z^{-\frac{n+2}{2}}\notag
\\
&=2R_{ninj}y_iy_jZ^{-\frac{n}{2}}-6nR_{ninj}y_iy_jy_n^2Z^{-\frac{n+2}{2}}\notag
\\
&\hspace{1cm}-4nR_{ninj}y_iy_jy_nZ^{-\frac{n+2}{2}}\,.\notag
\end{align}
This proves the Claim.

Using the above claim,
\ba
\int_{\Be}(\Delta\phi)\phi dy&=2A^2\int_{\Be}\Z^{-n}R_{ninj}R_{nknl}y_iy_jy_ky_ly_n^2dy\notag
\\
&\hspace{0.5cm}-4nA^2\int_{\Be}\Z^{-n-1}R_{ninj}R_{nknl}y_iy_jy_ky_ly_n^3dy\notag
\\
&\hspace{0.5cm}-6nA^2\int_{\Be}\Z^{-n-1}R_{ninj}R_{nknl}y_iy_jy_ky_ly_n^4dy\,.\notag
\end{align}

Since $\Delta^2(R_{ninj}R_{nknl}y_iy_jy_ky_l)=16(R_{ninj})^2$,  
$$
\int_{S_r^{n-2}}R_{ninj}R_{nknl}y_iy_jy_ky_l d\sigma_r 
=\frac{2\sigma_{n-2}}{(n+1)(n-1)}r^{n+2}(R_{ninj})^2\,.
$$
Thus,
\ba
\int_{\Be}(\Delta\phi)\phi dy
=&\frac{4}{(n+1)(n-1)}A^2(R_{ninj})^2
\int_{B_{\delta\e^{-1}}^+}\frac {y_n^2|\bar{y}|^{4}}{\Z^n}dy\notag
\\
&\hspace{0.5cm}-\frac{8n}{(n+1)(n-1)}A^2(R_{ninj})^2
\int_{B_{\delta\e^{-1}}^+}\frac{y_n^3|\bar{y}|^{4}}{\Z^{n+1}}dy\notag
\\
&\hspace{0.5cm}-\frac{12n}{(n+1)(n-1)}A^2(R_{ninj})^2
\int_{B_{\delta\e^{-1}}^+}\frac{y_n^4|\bar{y}|^{4}}{\Z^{n+1}}dy\,.\notag
\end{align}

Hence,
\ba
\int_{\B}|\nabla\phi_{\e}|^2 dx
\leq &-\frac{4}{(n+1)(n-1)}\e^4A^2(R_{ninj})^2
\int_{B_{\delta\e^{-1}}^+}\frac {y_n^2|\bar{y}|^{4}}{\Z^n}dy\notag
\\
&\hspace{0.5cm}+\frac{8n}{(n+1)(n-1)}\e^4A^2(R_{ninj})^2
\int_{B_{\delta\e^{-1}}^+}\frac{y_n^3|\bar{y}|^{4}}{\Z^{n+1}}dy\notag
\\
&\hspace{0.5cm}+\frac{12n}{(n+1)(n-1)}\e^4A^2(R_{ninj})^2
\int_{B_{\delta\e^{-1}}^+}\frac{y_n^4|\bar{y}|^{4}}{\Z^{n+1}}dy\notag
\\
&\hspace{0.5cm}+C\e^{n-2}\delta^{2-n}\,.\notag
\end{align}
\ep   
\begin{lemma}\label{lema2}
We have,
\ba
\int_{\B}(g^{ij}-&\delta^{ij})\d_i(\U+\phi_{\e})\d_j(\U+\phi_{\e})dx
=\notag
\\
&\hspace{0.9cm}\frac{(n-2)^2}{(n+1)(n-1)}\e^4R_{ninj;\,ij}
\int_{B_{\delta\e^{-1}}^+}\frac{y_n^2|\bar{y}|^{4}}{\Z^n}dy\notag
\\
&\hspace{0.5cm}+\frac{(n-2)^2}{2(n-1)}\e^4(R_{ninj})^2
\int_{B_{\delta\e^{-1}}^+}\frac{y_n^4|\bar{y}|^{2}}{\Z^n}dy\notag
\\
&\hspace{0.5cm}-\frac{4n(n-2)}{(n+1)(n-1)}\e^4A(R_{ninj})^2
\int_{B_{\delta\e^{-1}}^+}\frac{y_n^4|\bar{y}|^{2}}{\Z^{n+1}}dy+E_1\,,\notag
\end{align}
where
\begin{equation}
E_1=
\begin{cases}\notag
O(\e^4\delta^{-4})&\text{if}\:n=6\,,
\\
O(\e^5\log (\delta\e^{-1}))&\text{if}\:n=7\,,
\\
O(\e^5)&\text{if}\:n\geq 8\,.
\end{cases}
\end{equation}
\end{lemma}
\bp   

Observe that
\ba\label{lema2:1}
\int_{\B}(g^{ij}-\delta^{ij})&\d_i(\U+\phi_{\e})\d_j(\U+\phi_{\e})dx
=\int_{\B}(g^{ij}-\delta^{ij})\d_i\U\d_j\U dx\notag
\\
&\hspace{0.5cm}+2\int_{\B}(g^{ij}-\delta^{ij})\d_i\U\d_j\phi_{\e} dx
+\int_{\B}(g^{ij}-\delta^{ij})\d_i\phi_{\e}\d_j\phi_{\e} dx\,.
\end{align}

We will handle separately the three terms in the right hand side of (\ref{lema2:1}). The first term is
\ba
\int_{\B}(g^{ij}-\delta^{ij})(x)\d_i\U(x)&\d_j\U(x) dx
=\int_{\Be}(g^{ij}-\delta^{ij})(\e y)\d_iU(y)\d_jU(y) dy\notag
\\
&=(n-2)^2\int_{\Be}\Z^{-n}(g^{ij}-\delta^{ij})(\e y)y_iy_j dy\,.\notag
\end{align}
Hence, using Lemma \ref{exp_g_A} we obtain
\ba
\int_{\B}(g^{ij}-\delta^{ij})(x)\d_i\U(x)&\d_j\U(x) dx
=\notag
\\
&\hspace{0.9cm}\frac{(n-2)^2}{(n+1)(n-1)}\e^4R_{ninj;ij}
\int_{B_{\delta\e^{-1}}^+}\frac{y_n^2|\bar{y}|^{4}}{\Z^n}dy\notag
\\
&\hspace{0.5cm}+\frac{(n-2)^2}{2(n-1)}\e^4(R_{ninj})^2
\int_{B_{\delta\e^{-1}}^+}\frac{y_n^4|\bar{y}|^{2}}{\Z^n}dy+E'_1\,,\notag
\end{align}
where
\begin{equation}
E'_1=
\begin{cases}\notag
O(\e^4\delta)&\text{if}\:n=6\,,
\\
O(\e^5\log (\delta\e^{-1}))&\text{if}\:n=7\,,
\\
O(\e^5)&\text{if}\:n\geq 8\,.
\end{cases}
\end{equation}
The second term 
is
\ba
2&\int_{\B}(g^{ij}-\delta^{ij})(x)\d_i\U(x)\d_j\phi_{\e}(x) dx\notag
\\
&=-2\int_{\B}(g^{ij}-\delta^{ij})(x)\d_i\d_j\U(x)\phi_{\e}(x)dx\notag
-2\int_{\B}(\d_ig^{ij})(x)\d_j\U(x)\phi_{\e}(x) dx\notag
\\
&\hspace{2cm}+O(\e^{n-2}\delta^{2-n})\notag
\\
&=-2\e^{2}\int_{\Be}(g^{ij}-\delta^{ij})(\e y)\d_i\d_jU(y)\phi(y)dy
-2\e^3\int_{\Be}(\d_ig^{ij})(\e y)\d_jU(y)\phi(y) dy\notag
\\
&\hspace{2cm}
+O(\e^{n-2}\delta^{2-n})\,.
\end{align}
But,
\ba\label{exist:aux:1}
-2\e^2\int_{\Be}&(g^{ij}-\delta^{ij})(\e y)\d_i\d_jU(y)\phi(y) dy\notag
\\
&=-2(n-2)\e^2A\int_{\Be}\Z^{-n-1}(g^{ij}-\delta^{ij})(\e y)
\\
&\hspace{3.5cm}\cdot\left\{ny_iy_j-\Z\delta_{ij}\right\}R_{nknl}y_ky_ly_n^2dy\notag
\\
&=-\frac{4n(n-2)}{(n+1)(n-1)}\e^4A(R_{ninj})^2
\int_{B_{\delta\e^{-1}}^+}\frac{y_n^4|\bar{y}|^{4}}{\Z^{n+1}}dy+E'_2\,,\notag
\end{align}
where
\begin{equation}
E'_2=
\begin{cases}\notag
O(\e^4\delta)&\text{if}\:n=6\,,
\\
O(\e^5\log (\delta\e^{-1}))&\text{if}\:n=7\,,
\\
O(\e^5)&\text{if}\:n\geq 8\,.
\end{cases}
\end{equation}
In the last equality of \ref{exist:aux:1}, we used Lemma \ref{exp_g_B} and the fact that Lemma \ref{exp:g}, together with Lemma \ref{conseq:fermi:conf'}(i),(ii),(iii), implies
$$
\int_{S^{n-2}_r}(g^{ij}-\delta^{ij})(\e y)\delta_{ij}R_{nknl}y_ky_ld\sigma_r(y)
=\int_{S^{n-2}_r}O(\e^4|y|^4)R_{nknl}y_ky_ld\sigma_r(y)\,.
$$

We also have, by  Lemma \ref{exp:g} and Lemma \ref{conseq:fermi:conf'}(i),
\begin{equation}
-2\e^3\int_{\Be}(\d_ig^{ij})(\e y)\d_jU(y)\phi(y) dy\notag
\\
=E'_3=
\begin{cases}\notag
O(\e^4\delta)&\text{if}\:n=6\,,
\\
O(\e^5\log (\delta\e^{-1}))&\text{if}\:n=7\,,
\\
O(\e^5)&\text{if}\:n\geq 8\,.
\end{cases}
\end{equation}

Hence,
\ba
2\int_{\B}&(g^{ij}-\delta^{ij})(x)\d_i\U(x)\d_j\phi_{\e}(x) dx
=E'_2+E'_3\notag
\\
&\hspace{2cm}-\frac{4n(n-2)}{(n+1)(n-1)}\e^4A(R_{ninj})^2
\int_{B_{\delta\e^{-1}}^+}\frac{y_n^4|\bar{y}|^{4}}{\Z^{n+1}}dy\,.\notag
\end{align}

Finally, the third term in the right hand side of (\ref{lema2:1}) is written as
\ba
\int_{\B}(g^{ij}-\delta^{ij})(x)\d_i\phi_{\e}(x)\d_j\phi_{\e}(x)dx
&=\e^4\int_{\Be}(g^{ij}-\delta^{ij})(\e y)\d_i\phi(y)\d_j\phi(y)dy\notag
\\
&=\begin{cases}\notag
O(\e^4\delta)&\text{if}\:n=6\,,
\\
O(\e^5\log (\delta\e^{-1}))&\text{if}\:n=7\,,
\\
O(\e^5)&\text{if}\:n\geq 8\,.
\end{cases}
\end{align}

The result now follows if we choose $\e$ small such that $\log(\delta\e^{-1})>\delta^{2-n}$.
\ep   
\begin{lemma}\label{lema3}
We have,
\ba
\frac{n-2}{4(n-1)}&\int_{\B}R_g(\U+\phi_{\e})^2dx
=\frac{n-2}{8(n-1)}\e^4R_{;\,nn}
\int_{\Be}\frac{y_n^2}{\Z^{n-2}}dy\notag
\\
&\hspace{2cm}-\frac{n-2}{24(n-1)^2}\e^4(\bar{W}_{ijkl})^2
\int_{\Be}\frac{|\bar{y}|^{2}}{\Z^{n-2}}dy+E_2\,,\notag
\end{align}
\end{lemma}
where
\begin{equation}
E_2=
\begin{cases}\notag
O(\e^4\delta)&\text{if}\:n=6\,,
\\
O(\e^5\log (\delta\e^{-1}))&\text{if}\:n=7\,,
\\
O(\e^5)&\text{if}\:n\geq 8\,.
\end{cases}
\end{equation}
\bp 
We first observe that
\begin{equation}\label{exist:aux:2}
\int_{\B}R_g(\U+\phi_{\e})^2dx
=\int_{\B}R_g\U^2dx+2\int_{\B}R_g\U\phi_{\e} dx+\int_{\B}R_g\phi_{\e}^2dx\,.
\end{equation}

We will handle each term in the right hand side of (\ref{exist:aux:2}) separately. Using Lemma \ref{exp_g_C}, we see that the first term is
\ba
\int_{\B}R_g(x)\U(x)^2dx
&=\e^2\int_{\Be}R_g(\e y)U^2(y)dy\notag
\\
\\
&=\frac{1}{2}\e^4R_{;\,nn}
\int_{\Be}\frac{y_n^2}{\Z^{n-2}}dy
+E'_4\notag
\\
&\hspace{0.5cm}-\frac{1}{12(n-1)}\e^4(\bar{W}_{ijkl})^2
\int_{\Be}\frac{|\bar{y}|^{2}}{\Z^{n-2}}dy\,,\notag
\end{align}
where 
\begin{equation}
E'_4=
\begin{cases}\notag
O(\e^4\delta)&\text{if}\:n=6\,,
\\
O(\e^5\log (\delta\e^{-1}))&\text{if}\:n=7\,,
\\
O(\e^5)&\text{if}\:n\geq 8\,.
\end{cases}
\end{equation}

By Lemma \ref{conseq:fermi:conf'}(ix), the second term is
\ba
2\int_{\B}R_g(x)\U(x)\phi_{\e}(x)dx
&=2\e^4\int_{\Be}R_g(\e y)U(y)\phi(y)dy\notag
\\
&=\begin{cases}\notag
O(\e^4\delta)&\text{if}\:n=6\,,
\\
O(\e^5\log (\delta\e^{-1}))&\text{if}\:n=7\,,
\\
O(\e^5)&\text{if}\:n\geq 8
\end{cases}
\end{align}
and the last term is
\begin{equation}
\int_{\B}R_g\phi_{\e}^2dx=
\begin{cases}\notag
O(\e^4\delta)&\text{if}\:n=6\,,
\\
O(\e^5\log (\delta\e^{-1}))&\text{if}\:n=7\,,
\\
O(\e^5)&\text{if}\:n\geq 8\,.
\end{cases}
\end{equation}
\ep   


\subsection{Proof of Theorem \ref{exist:thm}}

Now, we proceed to the proof of Theorem \ref{exist:thm}.
\begin{proof}[Proof of Theorem \ref{exist:thm}]
It folows from  Lemmas \ref{lema1}, \ref{lema2} and \ref{lema3} and the identities (\ref{exist:1}), (\ref{exist:2}) and (\ref{exist:3}) that
\begin{align}\label{exist:4}
E_{M}(\psi)&\leq Q(B^n,\d B^n)\left(\int_{\d M}\psi^{\frac{2(n-1)}{n-2}}\right)^{\frac{n-2}{n-1}}
+E\notag
\\
&\hspace{0.5cm}-\e^4
\frac{4A^2}{(n+1)(n-1)}(R_{ninj})^2
\int_{\Be}\frac{y_n^2|\bar{y}|^{4}}{\Z^n}dy\notag
\\
&\hspace{0.5cm}+\e^4
\frac{(n-2)^2}{(n+1)(n-1)}R_{ninj;\,ij}
\int_{\Be}\frac{y_n^2|\bar{y}|^{4}}{\Z^n}dy\notag
\\
&\hspace{0.5cm}+\e^4
\frac{8nA^2}{(n+1)(n-1)}(R_{ninj})^2\int_{\Be}\frac{y_n^3|\bar{y}|^{4}}{\Z^{n+1}}dy\notag
\\
&\hspace{0.5cm}+\e^4
\frac{12nA^2}{(n+1)(n-1)}
(R_{ninj})^2\int_{\Be}\frac{y_n^4|\bar{y}|^{4}}{\Z^{n+1}}dy\notag
\\
&\hspace{0.5cm}-\e^4
\frac{4n(n-2)A}{(n+1)(n-1)}
(R_{ninj})^2\int_{\Be}\frac{y_n^4|\bar{y}|^{4}}{\Z^{n+1}}dy\notag
\\
&\hspace{0.5cm}+\e^4
\frac{(n-2)^2}{2(n-1)}(R_{ninj})^2\int_{\Be}\frac{y_n^4|\bar{y}|^{2}}{\Z^{n}}dy\notag
\\
&\hspace{0.5cm}+\e^4\frac{n-2}{8(n-1)}
R_{\,;\,nn}\int_{\Be}\frac{y_n^2}{\Z^{n-2}}dy\notag
\\
&\hspace{0.5cm}-\e^4\frac{n-2}{48(n-1)^2}
(\bar{W}_{ijkl})^2\int_{\Be}\frac{|\bar{y}|^{2}}{\Z^{n-2}}dy\,.
\end{align}
where
\begin{equation}
E=
\begin{cases}\notag
O(\e^4\delta^{-4})&\text{if}\:n=6\,,
\\
O(\e^5\log (\delta\e^{-1}))&\text{if}\:n=7\,,
\\
O(\e^5)&\text{if}\:n\geq 8\,.
\end{cases}
\end{equation}

We divide the rest of the proof in two cases.
\\\\
{\underline{The case $n=7,8$.}} 

Set $I=\int_0^{\infty}\frac{r^n}{(r^2+1)^n}dr$. We will apply the change of variables $\bar{z}=(1+y_n)^{-1}\bar{y}$ and Lemmas \ref{int:partes} and \ref{rel:int:t} in order to compare the different integrals in the expansion (\ref{exist:4}).

These integrals are
\ba
I_1=\int_{\Rn}\frac{y_n^2|\bar{y}|^{4}}{\Zt^{n}}dy_nd\bar{y}
&=\int_{0}^{\infty}y_n^2(1+y_n)^{3-n}dy_n\int_{\Rnn}\frac{|\bar{z}|^{4}}{(1+|\bar{z}|^2)^n}d\bar{z}\notag
\\
&=\frac{2(n+1)\,\sigma_{n-2}\,I}{(n-3)(n-4)(n-5)(n-6)}\,,\notag
\end{align}
\ba
I_2=\int_{\Rn}\frac{y_n^3|\bar{y}|^{4}}{\Zt^{n+1}}dy_nd\bar{y}
&=\int_{0}^{\infty}y_n^3(1+y_n)^{1-n}dy_n\int_{\Rnn}\frac{|\bar{z}|^{4}}{(1+|\bar{z}|^2)^{n+1}}d\bar{z}\notag
\\
&=\frac{3(n+1)\,\sigma_{n-2}\,I}{n(n-2)(n-3)(n-4)(n-5)}\,,\notag
\end{align}
\ba
I_3=\int_{\Rn}\frac{y_n^4|\bar{y}|^{4}}{\Zt^{n+1}}dy_nd\bar{y}
&=\int_{0}^{\infty}y_n^4(1+y_n)^{1-n}dy_n\int_{\Rnn}\frac{|\bar{z}|^{4}}{(1+|\bar{z}|^2)^{n+1}}d\bar{z}\notag
\\
&=\frac{12(n+1)\,\sigma_{n-2}\,I}{n(n-2)(n-3)(n-4)(n-5)(n-6)}\,,\notag
\end{align}
\ba
I_4=\int_{\Rn}\frac{y_n^4|\bar{y}|^{2}}{\Zt^{n}}dy_nd\bar{y}
&=\int_{0}^{\infty}y_n^4(1+y_n)^{1-n}dy_n\int_{\Rnn}\frac{|\bar{z}|^{2}}{(1+|\bar{z}|^2)^{n}}d\bar{z}\notag
\\
&=\frac{24\,\sigma_{n-2}\,I}{(n-2)(n-3)(n-4)(n-5)(n-6)}\notag
\end{align}
and
\ba
I_5=\int_{\Rn}\frac{y_n^2}{\Zt^{n-2}}dy_nd\bar{y}
&=\int_{0}^{\infty}y_n^2(1+y_n)^{3-n}dy_n\int_{\Rnn}\frac{1}{(1+|\bar{z}|^2)^{n-2}}d\bar{z}\notag
\\
&=\frac{8(n-2)\,\sigma_{n-2}\,I}{(n-3)(n-4)(n-5)(n-6)}\,.\notag
\end{align}
Thus,
\begin{align}\label{exist:4'}
E_{M}(\psi)&\leq Q(B^n,\d B^n)\left(\int_{\d M}\psi^{\frac{2(n-1)}{n-2}}\right)^{\frac{n-2}{n-1}}
+E'\notag
\\
&\hspace{0.5cm}+\e^4
\left\{-\frac{4A^2}{(n+1)(n-1)}I_1+\frac{8nA^2}{(n+1)(n-1)}I_2+\frac{(n-2)^2}{2(n-1)}I_4\right\}
(R_{ninj})^2\notag
\\
&\hspace{0.5cm}+\e^4
\left\{\frac{12nA^2}{(n+1)(n-1)}-\frac{4n(n-2)A}{(n+1)(n-1)}\right\}I_3\cdot
(R_{ninj})^2\notag
\\
&\hspace{0.5cm}+\e^4\frac{(n-2)^2}{(n+1)(n-1)}I_1\cdot R_{ninj;\,ij}+\e^4\frac{n-2}{8(n-1)}I_5\cdot
R_{\,;\,nn}\notag
\\
&\hspace{0.5cm}-\e^4\frac{n-2}{48(n-1)^2}
(\bar{W}_{ijkl})^2\int_{\Rn}\frac{|\bar{y}|^{2}}{\Z^{n-2}}dy\,.
\end{align}
where
\begin{equation}
E'=
\begin{cases}\notag
O(\e^5\log (\delta\e^{-1}))&\text{if}\:n=7\,,
\\
O(\e^5)&\text{if}\:n= 8\,.
\end{cases}
\end{equation}

Using Lemma \ref{conseq:fermi:conf'}(xi) and substituting the expressions obtained for $I_1$,...,$I_5$ in the expansion (\ref{exist:4'}), the coefficients of $R_{ninj;\,ij}$ and $R_{\,;\,nn}$ cancel out and we obtain
\ba\label{exist:5}
E_{M}(\psi)&\leq Q(B^n,\d B^n)\left(\int_{\d M}\psi^{\frac{2(n-1)}{n-2}}\right)^{\frac{n-2}{n-1}}
+E'\notag
\\
&\hspace{0.5cm}+\e^4\sigma_{n-2}I\cdot \gamma\left\{16(n+1)A^2-48(n-2)A+2(8-n)(n-2)^2\right\}(R_{ninj})^2\notag
\\
&\hspace{0.5cm}-\e^4\frac{n-2}{48(n-1)^2}
(\bar{W}_{ijkl})^2\int_{\Rn}\frac{|\bar{y}|^{2}}{\Z^{n-2}}dy\,,
\end{align}
where
$$
\gamma=\frac{1}{(n-1)(n-2)(n-3)(n-4)(n-5)(n-6)}\,.
$$

Choosing $A=1$, the term $16(n+1)A^2-48(n-2)A+2(8-n)(n-2)^2$ in the expansion (\ref{exist:5}) is $-62$ for $n=7$ and $-144$ for $n=8$. Thus, for small $\e$, since $W_{abcd}(x_0)\neq 0$, the expansion (\ref{exist:5}) together with Lemma \ref{anul:W} implies that
$$
E_{M}(\psi)<Q(B^n,\d B^n)\left(\int_{\d M}\psi^{\frac{2(n-1)}{n-2}}\right)^{\frac{n-2}{n-1}}
$$
for dimensions $7$ and $8$.
\\\\
{\underline{The case $n=6$.}} 

We will again apply the change of variables $\bar{z}=(1+y_n)^{-1}\bar{y}$ and Lemma \ref{int:partes} in order to compare the different integrals in the expansion (\ref{exist:4}). In the next estimates we are always assuming $n=6$.

In this case, the first integral is
\ba
I_{1,\delta/\e}=\int_{\Be}\frac{y_n^2|\bar{y}|^{4}}{\Zt^{n}}dy_nd\bar{y}
&=\int_{\Be\cap\{y_n\leq\delta/2\e\}}\frac{y_n^2|\bar{y}|^{4}}{\Zt^{n}}dy_nd\bar{y}+O(1)\notag
\\
&=\int_{\Rn\cap\{y_n\leq\delta/2\e\}}\frac{y_n^2|\bar{y}|^{4}}{\Zt^{n}}dy_nd\bar{y}+O(1)\,.\notag
\end{align}
Hence,
\ba
I_{1,\delta/\e}
&=\int_{0}^{\delta/2\e}y_n^2(1+y_n)^{3-n}dy_n\int_{\Rnn}\frac{|\bar{z}|^{4}}{(1+|\bar{z}|^2)^n}d\bar{z}
+O(1)\notag
\\
&=\log (\delta\e^{-1})\frac{n+1}{n-3}\sigma_{n-2}\,I+O(1)\,.\notag
\end{align}

The second integral is
$$
I_{2,\delta/\e}=\int_{\Be}\frac{y_n^3|\bar{y}|^{4}}{\Zt^{n+1}}dy_nd\bar{y}=O(1)\,.
$$

Similarly to $I_{1,\delta/\e}$, the others integrals are
\ba
I_{3,\delta/\e}&=\int_{\Be}\frac{y_n^4|\bar{y}|^{4}}{\Zt^{n+1}}dy_nd\bar{y}\notag
\\
&=\int_{0}^{\delta/2\e}y_n^4(1+y_n)^{1-n}dy_n\int_{\Rnn}\frac{|\bar{z}|^{4}}{(1+|\bar{z}|^2)^{n+1}}d\bar{z}
+O(1)\notag
\\
&=\log (\delta\e^{-1})\frac{n+1}{2n}\sigma_{n-2}\,I+O(1)\,,\notag
\end{align}
\ba
I_{4,\delta/\e}&=\int_{\Be}\frac{y_n^4|\bar{y}|^{2}}{\Zt^{n}}dy_nd\bar{y}\notag
\\
&=\int_{0}^{\delta/2\e}y_n^4(1+y_n)^{1-n}dy_n\int_{\Rnn}\frac{|\bar{z}|^{2}}{(1+|\bar{z}|^2)^{n}}d\bar{z}\notag
\\
&=\log (\delta\e^{-1})\sigma_{n-2}\,I+O(1)\,,\notag
\end{align}
\ba
I_{5,\delta/\e}&=\int_{\Be}\frac{y_n^2}{\Zt^{n-2}}dy_nd\bar{y}\notag
\\
&=\int_{0}^{\delta/2\e}y_n^2(1+y_n)^{3-n}dy_n\int_{\Rnn}\frac{1}{(1+|\bar{z}|^2)^{n-2}}d\bar{z}+O(1)\notag
\\
&=\log (\delta\e^{-1})\frac{4(n-2)}{n-3}\sigma_{n-2}\,I+O(1)\notag
\end{align}
and
\ba
I_{6,\delta/\e}&=\int_{\Be}\frac{|\bar{y}|^{2}}{\Z^{n-2}}dy_nd\bar{y}\notag
\\
&=\int_{0}^{\delta/2\e}(1+y_n)^{5-n}dy_n\int_{\Rnn}\frac{|\bar{z}|^2}{(1+|\bar{z}|^2)^{n-2}}d\bar{z}+O(1)\notag
\\
&=\log (\delta\e^{-1})\frac{4(n-1)(n-2)}{(n-3)(n-5)}\sigma_{n-2}\,I+O(1)\,.\notag
\end{align}
Thus,
\begin{align}\label{exist:6}
E_{M}(\psi)&\leq Q(B^n,\d B^n)\left(\int_{\d M}\psi^{\frac{2(n-1)}{n-2}}\right)^{\frac{n-2}{n-1}}
+O(\e^4\delta^{-4})\notag
\\
&\hspace{0.5cm}+\e^4
\left\{-\frac{4A^2}{(n+1)(n-1)}I_{1,\delta/\e}
+\frac{(n-2)^2}{2(n-1)}I_{4,\delta/\e}\right\}
(R_{ninj})^2\notag
\\
&\hspace{0.5cm}+\e^4
\left\{\frac{12nA^2}{(n+1)(n-1)}-\frac{4n(n-2)A}{(n+1)(n-1)}\right\}I_{3,\delta/\e}\cdot
(R_{ninj})^2\notag
\\
&\hspace{0.5cm}+\e^4\frac{(n-2)^2}{(n+1)(n-1)}I_{1,\delta/\e}\cdot R_{ninj;\,ij}
+\e^4\frac{n-2}{8(n-1)}I_{5,\delta/\e}\cdot R_{\,;\,nn}\notag
\\
&\hspace{0.5cm}-\e^4\frac{n-2}{48(n-1)^2}I_{6,\delta/\e}\cdot(\bar{W}_{ijkl})^2\,.
\end{align}

Using Lemma \ref{conseq:fermi:conf'}(xi) and substituting the expressions obtained for $I_{1,\delta/\e}$,...,$I_{6,\delta/\e}$ in expansion (\ref{exist:6}), the coefficients of $R_{ninj;\,ij}$ and $R_{\,;\,nn}$ cancel out and we obtain
\ba\label{exist:7}
E_{M}(\psi)&\leq Q(B^n,\d B^n)\left(\int_{\d M}\psi^{\frac{2(n-1)}{n-2}}\right)^{\frac{n-2}{n-1}}
+O(\e^4\delta^{-4})\notag
\\
&\hspace{0.5cm}+\e^4\log (\delta\e^{-1})\sigma_{n-2}I\cdot\notag
\\
&\hspace{1cm}\left\{\frac{6(n-3)-4}{(n-1)(n-3)}A^2-\frac{2(n-2)}{n-1}A+\frac{(n-2)^2(n-5)}{2(n-1)(n-3)}\right\}(R_{ninj})^2\notag
\\
&\hspace{0.5cm}-\e^4\log (\delta\e^{-1})\sigma_{n-2}I\frac{(n-2)^2}{12(n-1)(n-3)(n-5)}
(\bar{W}_{ijkl})^2\,.
\end{align}

Choosing $A=1$, the term $\frac{6(n-3)-4}{(n-1)(n-3)}A^2-\frac{2(n-2)}{n-1}A+\frac{(n-2)^2(n-5)}{2(n-1)(n-3)}$ in the expansion (\ref{exist:7}) is $-\frac{2}{15}$ for $n=6$. Thus, for small $\e$, since $W_{abcd}(x_0)\neq 0$, the expansion (\ref{exist:7}) together with Lemma \ref{anul:W} implies that
$$
E_{M}(\psi)<Q(B^n,\d B^n)\left(\int_{\d M}\psi^{\frac{2(n-1)}{n-2}}\right)^{\frac{n-2}{n-1}}
$$
for dimension $n=6$.
\end{proof}


\renewcommand{\theequation}{A-\arabic{equation}}
\setcounter{equation}{0}
\renewcommand{\thetheorem}{A-\arabic{theorem}}
\setcounter{theorem}{0}
\section*{Appendix A}

In this section, we will use the results of Section 2 to calculate some integrals used in the computations of Section 3. We recall that all curvature coefficients are evaluated at $x_0\in \d M$ and we are making use of conformal Fermi coordinates centered at this point.
\begin{lemma}\label{exp_g_A}
We have
\begin{align}
\int_{S^{n-2}_r}(g^{ij}-\delta^{ij})(\epsilon y)y_iy_jd\sigma_r(y)
&=\sigma_{n-2}\e^4\frac{y_n^2r^{n+2}}{(n+1)(n-1)}R_{ninj;\,ij}\notag
\\
&\hspace{0.5cm}+\sigma_{n-2}\e^4\frac{y_n^4r^{n}}{2(n-1)}(R_{ninj})^2+O(\e^5|(r,y_n)|^{n+5})\,.\notag
\end{align}
\end{lemma}
\bp
By Lemma \ref{exp:g},
\begin{align}
\int_{S^{n-2}_r}(g^{ij}-\delta^{ij})(\epsilon y)y_iy_jd\sigma_r(y)
&=\e^4\int_{S_r^{n-2}}\frac{1}{2}R_{ninj;\,kl}y_iy_jy_ky_ld\sigma_r(y)+O(\e^5|(r,y_n)|^{n+5})\notag
\\
&+\e^4y_n^2\int_{S_r^{n-2}}(\frac{1}{12}R_{ninj;\,nn}d\sigma_r(y)
+\frac{2}{3}R_{nins}R_{nsnj})y_iy_jd\sigma_r(y)\,.\notag
\end{align}
Then we just use the identity (\ref{rel:int:pol'}), Lemma \ref{conseq:fermi:conf'} and the fact that
$$\Delta^2(R_{ninj;\,kl}y_iy_jy_ky_l)=16R_{ninj;\,ij}\,.$$
\ep
\begin{lemma}\label{exp_g_B}
We have
\begin{align}
\int_{S^{n-2}_r}(g^{ij}-\delta^{ij})(\e y)R_{nknl}y_iy_jy_ky_ld\sigma_r(y)
&=\frac{2}{(n+1)(n-1)}\sigma_{n-2}\e^2y_n^2r^{n+2}(R_{ninj})^2\notag
\\
&\hspace{1cm}+O(\e^5|(r,y_n)|^{n+5})\notag
\end{align}
\end{lemma}
\bp
As in Lemma \ref{exp_g_A}, the result follows from 
\begin{align}
\int_{S^{n-2}_r}(g^{ij}-\delta^{ij})(\e y)R_{nknl}y_iy_jy_ky_ld\sigma_r(y)
&=\e^2y_n^2\int_{S^{n-2}_r}R_{ninj}R_{nknl}y_iy_jy_ky_ld\sigma_r(y)\notag
\\
&\hspace{1cm}+O(\e^5|(r,y_n)|^{n+5})\,,\notag
\end{align}
the fact that $\Delta^2(R_{ninj}R_{nknl}y_iy_jy_ky_l)=16(R_{ninj})^2$ and the identity (\ref{rel:int:pol'}).
\ep
\begin{lemma}\label{exp_g_C}
We have
\ba
\int_{S^{n-2}_r}R_g(\e y)d\sigma_r(y)
&=\sigma_{n-2}\e^2 \left\{\frac{1}{2}y_n^2r^{n-2}R_{;\,nn}
-\frac{1}{12(n-1)}r^{n}(\bar{W}_{ijkl})^2\right\}+O(\e^3|(r,y_n)|^{n+1})\,.\notag
\end{align}
\end{lemma}
\bp
As in Lemma \ref{exp_g_A}, the result follows from 
\ba
\int_{S^{n-2}_r}R_g(\e y)d\sigma_r(y)&=\e^2y_n^2\int_{S^{n-2}_r}\frac{1}{2}R_{;\,nn}d\sigma_r(y)
+\e^2\int_{S^{n-2}_r}\frac{1}{2}R_{;\,ij}y_iy_jd\sigma_r(y)\notag
\\
&\hspace{1cm}+O(\e^3|(r,y_n)|^{n+1})\,,\notag
\end{align}
Lemma \ref{conseq:fermi:conf'}(x) and the identity (\ref{rel:int:pol'}).
\ep


\renewcommand{\theequation}{B-\arabic{equation}}
\setcounter{equation}{0}
\renewcommand{\thetheorem}{B-\arabic{theorem}}
\setcounter{theorem}{0}
\section*{Appendix B}

In this section we will perform some integrations by parts that were used in the computations of Section 3. 
\begin{lemma}\label{int:partes}
We have:
\\\\
(a) $\int_0^{\infty}\frac{s^{\alpha}ds}{(1+s^2)^m}=\frac{2m}{\alpha +1}\int_0^{\infty}\frac{s^{\alpha+2}ds}{(1+s^2)^{m+1}}\;$, for $\alpha +1<2m$;
\\\\
(b) $\int_0^{\infty}\frac{s^{\alpha}ds}{(1+s^2)^m}=\frac{2m}{2m-\alpha -1}\int_0^{\infty}\frac{s^{\alpha}ds}{(1+s^2)^{m+1}}\;$, for $\alpha +1<2m$;
\\\\
(c) $\int_0^{\infty}\frac{s^{\alpha}ds}{(1+s^2)^m}=\frac{2m-\alpha-3}{\alpha +1}\int_0^{\infty}\frac{s^{\alpha+2}ds}{(1+s^2)^{m}}\;$, for $\alpha +3<2m$.
\end{lemma}
\begin{proof}
Integrating by parts,
$$\int_0^{\infty}\frac{s^{\alpha+2}ds}{(1+s^2)^{m+1}}=\int_0^{\infty} s^{\alpha+1}\frac{s\,ds}{(1+s^2)^{m+1}}=
\frac{\alpha+1}{2m}\int_0^{\infty}\frac{s^{\alpha} ds}{(1+s^2)^{m}}\,,$$
for $\alpha+1<2m$, which proves the item (a).

The item (b) follows from the item (a) and from
$$\int_0^{\infty}\frac{s^{\alpha}ds}{(1+s^2)^{m}}=\int_0^{\infty}\frac{s^{\alpha}(1+s^2)}{(1+s^2)^{m+1}}ds
=\int_0^{\infty}\frac{s^{\alpha}ds}{(1+s^2)^{m+1}}+\int_0^{\infty}\frac{s^{\alpha+2}ds}{(1+s^2)^{m+1}}\,.$$

To prove the item (c), observe that, by the item (a), 
$$\int_0^{\infty}\frac{s^{\alpha}ds}{(1+s^2)^{m-1}}
=\frac{2(m-1)}{\alpha+1}\int_0^{\infty}\frac{s^{\alpha+2}ds}{(1+s^2)^{m}}\,,$$ 
for $\alpha+3<2m$.
But, by the item (b), we have  
$$\int_0^{\infty}\frac{s^{\alpha}ds}{(1+s^2)^{m-1}}
=\frac{2(m-1)}{2(m-1)-\alpha-1}\int_0^{\infty}\frac{s^{\alpha}ds}{(1+s^2)^{m}}\,.$$
\end{proof}
\begin{lemma}\label{rel:int:t}
For $m>k+1$,
$$
\int_{0}^{\infty}\frac{t^k}{(1+t)^m}dt
=\frac{k!}{(m-1)(m-2)...(m-1-k)}\,.
$$
\end{lemma}
\begin{proof}
Integrating by parts,
$$
\int_{0}^{\infty}t^{k-1}(1+t)^{1-m}dt
=\frac{m-1}{k}\int_{0}^{\infty}t^{k}(1+t)^{-m}dt\,.
$$
On the other hand,
$$
\int_{0}^{\infty}t^{k-1}(1+t)^{1-m}dt
=\int_{0}^{\infty}\frac{t^{k-1}(1+t)}{(1+t)^m}dt
=\int_{0}^{\infty}\frac{t^{k}}{(1+t)^m}dt
+\int_{0}^{\infty}\frac{t^{k-1}}{(1+t)^m}dt\,.
$$
Hence, 
$$
\int_{0}^{\infty}\frac{t^k}{(1+t)^{m}}dt
=\frac{k}{m-1-k}\int_{0}^{\infty}\frac{t^{k-1}}{(1+t)^{m}}dt\,.
$$
Now the result follows observing that
$
\int_{0}^{\infty}\frac{1}{(1+t)^{m}}dt=\frac{1}{m-1}\,.
$
\end{proof}


\bigskip\noindent
U{\small{NIVERSIDADE}} F{\small{EDERAL}} F{\small{LUMINENSE}} (UFF)\\ I{\small{NSTITUTO}} DE M{\small{ATEM\'{A}TICA}}\\
R{\small{UA}} M{\small{\'{A}RIO}} S{\small{ANTOS}} B{\small{RAGA S/N}}\\
N{\small{ITER\'{O}I,}} RJ 24020-140\\ B{\small{RAZIL}}\vspace{0.1cm}

\noindent
{\bf{almaraz@impa.br}}


\end{document}